\documentclass[12pt]{amsart}
\usepackage{amssymb,latexsym}
\usepackage{amsfonts}
\usepackage{amsmath}
\usepackage[colorlinks,linkcolor=blue,anchorcolor=blue,citecolor=blue]{hyperref}
\usepackage{algorithm}
\usepackage{enumerate}
\usepackage{algpseudocode}
\usepackage{verbatim}
\usepackage{graphicx}

\newcommand\F{{\mathbb F}}
\newcommand\Q{{\mathbb Q}}

\newcommand\Z{{\mathbb Z}}

\newtheorem{theorem}{Theorem}[section]
\newtheorem{lemma}[theorem]{Lemma}

\newtheorem{corollary}[theorem]{Corollary}

\newtheorem{conjecture}[theorem]{Conjecture}
\newtheorem{question}[theorem]{Question}

\theoremstyle{definition}

\newtheorem{remark}[theorem]{Remark}

\numberwithin{equation}{section}

\begin{document}

\title[]{The distance to square-free polynomials}

\author{Art\= uras Dubickas}
\address{Institute of Mathematics, Faculty of Mathematics and Informatics, Vilnius University, 
Naugarduko 24, Vilnius LT-03225,  Lithuania}
\email{arturas.dubickas@mif.vu.lt}

\author{Min Sha}
\address{Department of Computing, Macquarie University, Sydney, NSW 2109, Australia}
\email{shamin2010@gmail.com}





\begin{abstract} 
In this paper, we consider a variant of Tur{\' a}n's problem on the distance from an integer polynomial in $\Z[x]$ to the nea\-rest irreducible polynomial in $\Z[x]$. 
We prove that for any polynomial $f \in \Z[x]$, there exist infinitely many square-free polynomials $g\in \Z[x]$ 
such that $L(f-g) \le 2$, where $L(f-g)$ denotes the sum of the absolute values of the coefficients of $f-g$. On the other hand, we show that this inequality cannot be replaced by $L(f-g) \le 1$. 
For this, for each integer $d \geq 15$ 
we construct infinitely many polynomials
$f \in \Z[x]$ of degree $d$ such that neither $f$ itself nor any $f(x) \pm x^k$, where
$k$ is a non-negative integer, is square-free. Polynomials over prime fields and their distances to square-free polynomials are also considered. 
\end{abstract}

\maketitle

\section{Introduction}

For an integer polynomial $f(x)= a_dx^d + a_{d-1}x^{d-1} + \cdots + a_0\in \Z[x]$ of degree $d\ge 1$, 
its \textit{length} $L(f)$ is defined by 
$$
L(f) = |a_d| + |a_{d-1}| + \cdots + |a_0|, 
$$
and its \textit{height} $H(f)$ by  
$$
H(f) = \max \{|a_d|, |a_{d-1}|, \ldots, |a_0| \}.
$$
In 1960s, Tur{\' a}n \cite{Schinzel67} asked if
there exists an absolute constant $C$ such that for any polynomial $f \in \Z[x]$, 
there is an irreducible   (over the rational numbers) polynomial $g \in \Z[x]$ of degree at most $\deg f$ satisfying $L(f-g) \le C$.

Although Tur{\' a}n's problem remains open, a number of partial results have been obtained. See, for instance, a recent review of Filaseta \cite{filas}. 
In 1970, Schinzel \cite{Schinzel1970} proved that $C = 3$ suffices if one removes the condition on the degree of $g$. More precisely,
he showed that if $f \in \Z[x]$ is of degree $d$, then there are infinitely many irreducible polynomials $g \in \Z[x]$ such that 
\begin{equation*}
L(f-g) \le \left\{
\begin{array}{rl}
2 & \textrm{if $f(0) \ne 0$,}\\
3 & \textrm{always,}
\end{array} \right.
\end{equation*}
and, moreover, at least one of them satisfies 
 $$
 \deg g \le \exp((5d+7)(\|f\| +3)), 
 $$
where $\|f\|$ stands for the sum of the squares of the coefficients of $f$. 
In \cite{BF}, Banerjee and Filaseta improved the above upper bound to 
$$
 \deg g \le 8 \max\{d+3,c_0 \}5^{8\|f\| +9}, 
$$
where $c_0$ is an effectively computable absolute constant. 
In addition, using computational strategies, it has been confirmed in \cite{BH1,BH2,FM,LRW,Moss} that 
if $f \in \Z[x]$ has degree $d \le 40$, then there exists an irreducible polynomial $g \in \Z[x]$ with $\deg g = d$ and $L(f - g) \le 5$.
On the other hand, although the trivial example $f(x)=x^3$ shows that $C \geq 2$, it is not known that the optimal constant $C$ should be strictly greater than $2$. 

In this paper, we consider a variant of Tur{\' a}n's problem, where ``irreducible polynomial $g$" is replaced by ``square-free polynomial $g$". 
For this, we pose the following conjecture: 

\begin{conjecture}   \label{conj:Turan}
For any $f \in \Z[x]$ of degree $d$, 
there is a square-free  polynomial $g \in \Z[x]$ of degree at most $d$ satisfying $L(f-g) \le 2$.  
\end{conjecture}

Another problem related to Tur\'an's problem is that of Szegedy as\-king if there exists a constant $C_0$ depending only on $d$ such that for any
$f \in \Z[x]$ of degree $d$, the polynomial $f(x)+t$ is irreducible for some $t \in \Z$ with
$|t| \leq C_0$. In general, the problem of Szegedy is still open; see the papers of Gy\H{o}ry \cite{BH10} and Hajdu \cite{BH11}. 
However, in our setting, when ``irreducible" is replaced by ``square-free", this problem becomes very simple. 
One can take, for instance, $C_0=\lfloor d/2 \rfloor$.

\begin{theorem}\label{badd}
For any $f \in \Z[x]$ of degree $d$, at least one of the polynomials
$f(x)+t$, where $t \in \Z$ satisfies $|t| \leq \lfloor d/2 \rfloor$, 
is square-free. 
\end{theorem}

\begin{proof}
Let $S$ be a subset of $\Z$  with the property
that for each integer $t \in S$ some $h_t^2$, where $h_t \in \Z[x]$ is of degree at least $1$, divides the polynomial $f(x)+t$. Then, $h_t \ne h_s$ when $t \ne s$ both belong to $S$, since otherwise $h_t \mid (t-s)$, a contradiction. Also, $h_t$ divides the derivative $f^\prime$ for every $t \in S$, so the cardinality of the set $S$ does not exceed $\deg f^\prime \leq d-1$.
The assertion of the theorem now follows, because the set $\{-\lfloor d/2 \rfloor, \dots, 0, \dots, \lfloor d/2 \rfloor\}$ contains
at least $d$ integers. 
\end{proof}

Note that Theorem~\ref{badd} implies Conjecture~\ref{conj:Turan}
for polynomials of degree $d \leq 5$. Moreover, for $d=2$ and $d=3$, the inequality $L(f-g) \leq 2$ in Conjecture~\ref{conj:Turan} can be replaced by $L(f-g) \leq 1$. 
However, in general, the condition $L(f-g) \le 2$ of Conjecture \ref{conj:Turan} cannot be relaxed.

\begin{theorem}   \label{thm:L1}
For any integer $d \ge 15$, there exist infinitely many polynomials $f \in \Z[x]$ of degree $d$ 
such that each polynomial $g \in \Z[x]$ satisfying $L(f-g) \le 1$
is not square-free.
\end{theorem}

As one can see from the proof given in Section~\ref{sec:L1}, one example of such degree $15$ polynomials is 
\begin{equation}   \label{eq:example}
\begin{split}
f(x) = & 15552x^{15} + 5184x^{14} + 5616x^{13} + 8784x^{12} + 13908x^{11} \\
& + 13756x^{10} + 96413x^9 - 18929x^8 - 57229x^7 + 6851x^6 \\ 
& + 9435x^5 - 932x^4 - 346x^3 + 36x^2.
\end{split}
\end{equation}
Its root $0$ has multiplicity $2$. Also, $0$ is a root of multiplicity $2$ of any polynomial $f(x) \pm x^k$, where $k \geq 2$ is an integer, 
whereas $1/2, -1/2, 1/6$ and $-1/6$ are multiple roots of $f(x)-x, f(x)+1, f(x)+x$ and $f(x)-1$, respectively.  
We do not claim that $d=15$ is the smallest  degree of the polynomials satisfying the conditions of Theorem~\ref{thm:L1}.

In Section~\ref{sec:general}, we prove a weak form of Conjecture \ref{conj:Turan} by relaxing the condition on the degree of $g$: 

\begin{theorem}   \label{thm:general}
For any $f \in \Z[x]$ of degree $d$ and any integer 
\begin{equation}\label{klop}
n > L(f'), 
\end{equation}
there is a square-free polynomial $g\in \Z[x]$ satisfying $\deg g =n$ and 
\begin{equation*}
L(f-g) = \left\{
\begin{array}{rl}
1 & \textrm{if $x^2 \nmid f(x)$,}\\
2 & \textrm{always.}
\end{array} \right.
\end{equation*}
\end{theorem}

Note that for $f(x)=a_dx^d+a_{d-1}x^{d-1}+\dots+a_0$ one has 
\begin{align*}
L(f') & =d|a_d|+(d-1)|a_{d-1}|+\dots+|a_1| \\
& \leq \min \{dL(f), d(d+1)H(f)/2 \}, 
\end{align*}
so \eqref{klop} can be replaced by $n>dL(f)$ or $n>d(d+1)H(f)/2$.

Roughly speaking, the result in Theorem \ref{thm:general} confirms the existence of square-free polynomials $g$ close to $f$ with $\deg g$
arbitrary large. 
In the following theorem, we establish the existence of one square-free polynomial close to $f$ but of degree that for large $L(f)$ can be much smaller than the bound in \eqref{klop}. (In terms of $L(f)$, the bound $d L(f)$ on $\deg g$ is replaced by
the bound $2.2 d (\log d/\log \log d)^3 \log L(f)$.)

\begin{theorem}   \label{thm:special}
For any polynomial $f \in \Z[x]$ of degree $d \geq 3$, 
 there is a square-free polynomial $g\in \Z[x]$ satisfying 
 \begin{equation}\label{baab}
\deg g < 
\begin{cases}
2.2d \big(\log d/\log\log d\big)^3\log L(f) & \textrm{if $x^2 \nmid f(x)$,}\\
2.2d \big(\log d/\log\log d\big)^3\log (L(f)+1)  & \textrm{always,}
\end{cases}
\end{equation} 
and 
\begin{equation*}
L(f-g) = \left\{
\begin{array}{rl}
1 & \textrm{if $x^2 \nmid f(x)$,}\\
2 & \textrm{always.}
\end{array} \right.
\end{equation*}
\end{theorem}

The proof of Theorem~\ref{thm:special} is given in Section \ref{sec:special}. Then,
in Section~\ref{sec:binary} we confirm Conjecture \ref{conj:Turan} for several classes of integer polynomials by transfering our problem to binary polynomials,
that is, considering $f$ modulo $2$, 
and then using computational strategies.  
Finally, in Section~\ref{sec:prime} we consider polynomials over prime fields $\F_p$ and their  distances to square-free polynomials.

\section{Proof of Theorem~\ref{thm:L1}}
\label{sec:L1}

Observe that for any polynomial $f \in \Z[x]$ of the form $x^2h(x)$ with non-zero $h(x) \in \Z[x]$ (so that $f$ automatically is not square-free), 
if there were a square-free polynomial $g \in \Z[x]$ satisfying $L(f-g) \le 1$, 
then $g$ must be of the form $f(x) \pm 1$ or $f(x) \pm x$. 
So, our purpose is to find polynomials $f\in \Z[x]$ of the form $x^2h(x)$ such that none of the following four polynomials 
$$
f(x)+1, \quad f(x)-1, \quad  f(x)+x, \quad  f(x)-x
$$
 is square-free.   

Assume that 
\begin{equation}   \label{eq:congruence}
\begin{split}
& f \equiv 0 ~({\rm mod}~x^2), \quad  f \equiv  -1 ~({\rm mod}~(2x+1)^2), \\&  f \equiv  x ~({\rm mod}~(2x-1)^2), \quad 
 f \equiv  1 ~({\rm mod}~(6x+1)^2), \\&  f \equiv  -x ~({\rm mod}~(6x-1)^2).
\end{split}
\end{equation} 
Then, all the solutions in $\Z[x]$ of \eqref{eq:congruence} meet our purpose. 
By the Chinese Remainder Theorem and using PARI/GP \cite{PARI}, 
we obtain a solution $f_0 \in \Q[x]$ of \eqref{eq:congruence}: 
\begin{align*} 
f_0(x) = & 106515x^9 - 8991x^8 - \frac{236133}{4}x^7 + \frac{20385}{4}x^6 \\
& + \frac{152209}{16}x^5 - \frac{13701}{16}x^4 - \frac{22207}{64}x^3 + \frac{2243}{64}x^2. 
\end{align*}
Let $h(x)$ be the product of all five polynomials that appear in the moduli of \eqref{eq:congruence}. 
Then, 
\begin{align*}
h(x) = 20736x^{10} - 11520x^8 + 1888x^6 - 80x^4 + x^2.
\end{align*} 
So, the general solution of \eqref{eq:congruence} in $\Q[x]$ has the form 
$$
f = f_0 + hf_1, \quad f_1 \in \Q[x]. 
$$
Now, we want to choose suitable $f_1$ such that $f \in \Z[x]$.  

Notice that $f_0$ has six coefficients not in $\Z$. 
We then choose $f_1$ to be a polynomial in $\Q[x]$ of degree $5$: 
$$
f_1(x) = a_5x^5 + \cdots + a_1x + a_0 
$$
such that $f_0 + hf_1 \in \Z[x]$, that is, $hf_1$ is congruent to $-f_0$ modulo the integers. 
By comparing the coefficients modulo the integers starting from the lowest term, we obtain 
\begin{equation*}
\begin{split}
& a_0 \in \frac{61}{64} + \Z, \quad a_1 \in \frac{63}{64} + \Z,  \quad  a_2 \in \frac{9}{16} + \Z,  \\
& a_3 \in \frac{11}{16} + \Z,   \quad  a_4 \in \frac{1}{4} + \Z, \quad  a_5 \in \frac{3}{4} + \Z.
\end{split}
\end{equation*} 
This completes the proof of the theorem for $d=15$. 

In particular, choosing $a_0 = \frac{61}{64}$, $a_1 = \frac{63}{64}$, $a_2 = \frac{9}{16}$, 
$a_3 =\frac{11}{16}$, $a_4 = \frac{1}{4}$ and $a_5 = \frac{3}{4}$, we get the polynomial presented in \eqref{eq:example}. 

For $d\ge 16$, we first choose any polynomial $f(x)$ of degree $15$ as above (for instance, the polynomial in \eqref{eq:example}), 
and then consider the polynomial
$$
 f(x) + k (2x+1)^{2} (2x-1)^{2} (6x+1)^{2} (6x-1)^{2} x^{d-8}, 
$$
where $k$ is any non-zero integer. 
By the construction of $f(x)$,  we in fact complete the proof for $d \ge 16$.

\begin{remark}
The anonymous referee suggested the following approach to prove Theorem~\ref{thm:L1}. 
Starting with the polynomial list $[x^2]$, one may search for a squared polynomial of small degree 
which has resultant 1 with each polynomial in this list until one obtains a list of five polynomials. 
For example, here is a possible list: 
$$
[x^2, (x -1)^2 , (2x -1)^2 ,  (x^2 + x - 1)^2 ,  (x^3 - x^2 - 2x + 1)^2].
$$
Then, due to the resultants being 1, by solving the congruence equations (for instance, using PARI/GP \cite{PARI})
\begin{equation*}   
\begin{split}
& f \equiv 0 ~({\rm mod}~x^2), \quad  f \equiv  1 ~({\rm mod}~(x-1)^2), \quad  f \equiv  -1 ~({\rm mod}~(2x-1)^2), \\ 
& f \equiv  x ~({\rm mod}~(x^2+x-1)^2), \quad  f \equiv  -x ~({\rm mod}~(x^3-x^2-2x+1)^2), 
\end{split}
\end{equation*} 
one gets the following solution in $\Z[x]$ of degree 15:  
\begin{align*}
f(x)  = & 125200x^{15} - 325540x^{14} - 726388x^{13} + 2529575x^{12} + 552645x^{11} \\
& - 6814352x^{10} + 3701398x^9 + 6619994x^8 - 7934278x^7 \\&  + 313994x^6 
+ 3958516x^5 - 2649357x^4 + 723237x^3 - 74643x^2.  
\end{align*}
Then, Theorem~\ref{thm:L1} for $d \ge 16$ holds by taking 
$$
f(x)+k(x-1)^2 (2x-1)^2 (x^2 +x-1)^2 (x^3 -x^2 -2x+1)^2 x^{d-14},
$$
where $k$ is any non-zero integer.

More generally, one can use the list of polynomials
\begin{align*}
[x^2, (kx-1)^2 ,(2kx-1)^2 ,  (k^2x^2 +kx-1)^2 , (k^3x^3 -k^2x^2 -2kx+1)^2] 
\end{align*}
for any non-zero integer $k$, without affecting the fact that the resulting polynomial of degree 15 is in $\Z[x]$. 
To see that this indeed gives an inifinite list of polynomials $f(x)$ of degree 15 
with the desired property, it suffices to notice that  for any resulting polynomial $f$, the polynomial
$f(x)-1$ cannot be divisible by infinitely many polynomials of the form $(kx-1)^2$.   
\end{remark}

\section{Proof of Theorem \ref{thm:general}}
\label{sec:general}

We first assume that $x^2 \nmid f(x)$, where
$f(x)=a_dx^d+\dots+a_1x+a_0$. 
Consider
the polynomial $x^n+f(x)$ with any integer
$n>L(f')$. Here, $L(f')=d|a_d|+(d-1)|a_{d-1}|+\dots+|a_1| \geq d$.

Suppose $x^n+f(x)$ has a square factor $h(x)^2$, where $h$ is an irreducible monic polynomial. Let $\alpha$ be any root of this square factor. 
Clearly, $\alpha$ is a non-zero algebraic integer, because $h(x) \ne x$. In case $|\alpha|<1$ we replace $\alpha$ by its 
conjugate satisfying $|\alpha| \geq 1$. 
Since this (new) $\alpha$ of modulus at least 1 is a root of $nx^{n-1}+f^\prime (x)=0$, applying the inequa\-li\-ty 
$|f^\prime (\alpha)| \leq |\alpha|^{d-1}L(f')$  we obtain
$
n|\alpha|^{n-1} = |f^\prime (\alpha)| \leq |\alpha|^{d-1}L(f').
$
 Consequently, 
 $$
n \leq n|\alpha|^{n-d} \leq L(f'),
$$
contrary to the assumption \eqref{klop} on $n$. 
Hence, the polynomial $x^n+f(x)$ is square-free, which implies the required result.

Next, assume that $x^2 \mid f(x)$. 
Then, $f(0)=0$ and
applying the same argument as the above, we deduce that the polynomial $x^n + f(x) + 1$ is square-free. 
This completes the proof of the theorem.

\section{Proof of Theorem \ref{thm:special}}
\label{sec:special}

We first make some preparations. 
For any real number $r \ge 1$, let $\Phi(r)$ be the number of positive integers $n$ for which $\varphi(n) \le r$, where $\varphi$ is 
Euler's totient function.  
 Erd\H{o}s \cite{Erdos} has shown that 
 $$
 \Phi(r) \sim (\zeta(2)\zeta(3)/\zeta(6))r
 $$ 
 as $r \to \infty$, where $\zeta(s)$ is the Riemann zeta function and 
$$
\zeta(2)\zeta(3)/\zeta(6) = 1.943596\ldots.
$$
Based on Bateman's work \cite{Bateman}, Derbal \cite{Der} has given an explicit version: for $r \ge 240$, one has
\begin{equation}  \label{eq:Der}
\begin{split}
|\Phi(r) &-  (\zeta(2)\zeta(3)/\zeta(6))r |\\ 
& < 58.61r \exp \Big(-(\sqrt{2}/8)\sqrt{(\log r)(\log\log r)} \Big), 
\end{split}
\end{equation}
where $\log$ stands for the natural logarithm.  
Here, we present a simple and explicit upper bound for $\Phi(r)$.

\begin{lemma}  \label{lem:Phi}
For any real number $r \ge 1$, we have 
\begin{equation*}
\Phi(r) \le \left\{
\begin{array}{rl}
2.5r & \textrm{if $r \le 10^6$,}\\
23r & \textrm{always.}
\end{array} \right.
\end{equation*}
\end{lemma}

\begin{proof}
For $r > 10^6$,
by a direction computation,  from \eqref{eq:Der} we
derive that $\Phi(r)$ is less than
$$
(\zeta(2)\zeta(3)/\zeta(6))r  + 58.61r \exp \Big(-(\sqrt{2}/8)\sqrt{(\log r)(\log\log r)} \Big) < 23r. 
$$ 

To prove the bound $\Phi(r) \leq 2.5 r$ for $1 \le r \le 10^6$,
it suffices to establish this inequality for every integer $r$ between 1 and $10^6$. 
 For $r=1,2,3$ one has $\Phi(1)=2$ and 
$\Phi(2)=\Phi(3)=5$, so in the interval $r \in [1,4)$ the inequality
$\Phi(r) \leq 2.5r$ is true with equality for $r=2$. Now suppose that $r$ is an integer at least 4. 
Notice that $\Phi(r)=\Phi(r+1)$ when $r$ is even (because $\varphi(n)$ is even for any integer $n \ge 2$). 
So, we only need to establish the inequality for even integers $r$ at least 4. 
We make some computations to achieve this purpose.
In all what follows, we explain the algorithm. 

Note that, by \cite[Theorem 15]{Rosser}, for any positive integer $n\ge 3$ we have 
\begin{equation}   \label{eq:varphi1}
\varphi(n) > \frac{n}{1.782\log\log n + 2.507/\log\log n}. 
\end{equation}
Since   for any integer $n \ge 30000$ the inequality
$$
2.507/\log\log n < 0.461\log\log n,
$$
holds,
applying \eqref{eq:varphi1}, one gets the inequality
\begin{equation}   \label{eq:varphi2}
\varphi(n) > \frac{n}{2.243\log\log n} > n^{5/6}
\end{equation} 
for $n \ge 30000$.
Consequently, 
for any integer $n \ge 30000$, if $\varphi(n) \le r$, by \eqref{eq:varphi2},
we obtain
\begin{equation}    \label{eq:nr}
n < 2.243r\log\log n < 2.243r \log\log(r^{6/5}). 
\end{equation}

Now, given an even integer $r$ between 4 and $10^6$, we need to count positive integers $n$ satisfying $\varphi(n) \le r$.
By \eqref{eq:nr}, we only need to consider positive integers $n$ satisfying
$$
n \le \max\{30000, 2.243r \log\log(r^{6/5})\}.
$$ 
To speed up the computations, one can also use the fact that 
 $\varphi(2n)=\varphi(n)$ when $n$ is odd. 
By a direction computation (for instance, using PARI/GP \cite{PARI}), 
we have checked that $\Phi(r) \le 2.5 r$ for any integer $4 \le r \le 10^6$. 
\end{proof}

From Lemma \ref{lem:Phi}, it is natural  to conjecture that $\Phi(r) \le 2.5 r$ for any real $r \ge 1$. 
Moreover, there are only 37 positive integers $r \le 10^6$  for which the quotient $\Phi(r)/r$ is at least 2. 
We list them as follows: 
\begin{align*}
& 2, 4, 6, 8, 9, 10, 12, 13, 16, 18, 20, 24, 25, 26, 32, 36, 40, 42, 44, 48, 49, 50, \\
&72, 73, 74, 80, 84, 96, 97, 120, 121, 144, 145, 240, 241, 242, 288.
\end{align*}
The upper bound $2.5r$ is attained only when $r=2$. 
Besides, the upper bound $23r$ in Lemma \ref{lem:Phi} can be improved by more advanced computations. 
However, it is widely believed that computing the values of Euler's totient function is as hard as factoring positive integers. 
(What is already proved in Lemma~\ref{lem:Phi} is sufficient for our purposes and produces the same constant in Theorem \ref{thm:special}
as that with the optimal bound $\Phi(r) \leq 2.5 r$ for each $r \geq 1$.)

Now, we are ready to prove Theorem \ref{thm:special}.

\begin{proof}[Proof of Theorem \ref{thm:special}] 

We first assume that $x^2 \nmid f(x)$. 
Suppose that $h(x)^2$ divides $x^n+f(x)$ for some integer $n \ge 1$, where $h \in \Z[x]$ is an irreducible polynomial of degree at least $1$.
 Since $h$ divides both $x^n + f(x)$ and $nx^{n-1} + f^\prime(x)$,  
we derive that $h$ divides $nx^n+nf(x) - nx^n-xf^\prime(x)=nf(x)-xf'(x)$. 
It is easy to see that  the polynomial $nf(x) - xf^\prime(x)$ is non-zero. 
Hence, we have 
\begin{equation}    \label{eq:degh}
\deg h \leq d.
\end{equation}

Now, let us consider two cases: $h$ is a cyclotomic polynomial (Case C), and $h$ is not 
a cyclotomic polynomial (Case N). 

Case C.
We claim that $h^2$ divides at most one polynomial $x^n+f(x)$, where
$n \ge 1$ is an integer.  
Indeed, if it divides two such polynomials, say $x^m+f(x)$ and $x^n+f(x)$ (where $m>n$), then
$h^2$ also divides $x^{m-n}-1$, 
which contradicts with the fact that 
the polynomial $x^{m-n}-1$ only has simple roots. 
Now, since $h$ is a cycloto\-mic polynomial of degree at most $d$ (see \eqref{eq:degh}), 
the number of possibilities for $h$ does not exceed $\Phi(d)$, which is 
the number of positive integers $k$ with the property $\varphi(k) \le d$. 

Case N.
As above, assume that $h(x)^2$ divides $x^n+f(x)$ for some integer $n>d$. 
Then, $h$ is monic and $h(x) \ne x$.  
Suppose that $\alpha$ is the largest in modulus root of $h$. 
Note that $\alpha$ is an algebraic integer. 
 Since $h$ is not a cyclotomic polynomial, by Kronecker's theorem, $|\alpha|$ is strictly greater than 1. 
 Hence, from $\alpha^n=-f(\alpha)$, we deduce that
 $|\alpha|^{n} = |f(\alpha)| \leq |\alpha|^{d} L(f)$. 
 Consequently, 
 \begin{equation}\label{nji12}
 (n-d) \log |\alpha| \le \log L(f). 
 \end{equation}

Note that if $\alpha$ is a reciprocal algebraic integer,  we have 
$$
|\alpha|^{d/2} \ge M(\alpha),
$$
 where $M(\alpha)$ is the Mahler measure of $\alpha$. 
Then, using the lower bound for the Mahler 
measure
$$
\log M(\alpha) >  \frac{1}{4} \Big( \frac{\log\log d}{\log d} \Big)^{3}
$$ 
(see \cite[Theorem]{vout}, or \cite[Theorem 1]{Dob} for an earlier result)  
and \eqref{nji12}, we further obtain 
$$
\frac{n-d}{2d} \Big( \frac{\log\log d}{\log d} \Big)^{3}  < \frac{2(n-d)\log M(\alpha)}{d} \leq (n-d)\log |\alpha| \leq \log L(f). 
$$
Hence,
\begin{equation}   \label{eq:CaseN}
n < d+2d\Big( \frac{\log d}{\log\log d} \Big)^{3} \log L(f). 
\end{equation}
In case when $\alpha$ is nonreciprocal, by Smyth's result \cite{smy}, we have a stronger bound 
$|\alpha|^d \geq M(\alpha) \geq \theta =1.324\dots$ on $\log |\alpha|$ in \eqref{nji12}, 
where $\theta$ is the real root of $x^3-x-1=0$,  so \eqref{eq:CaseN} also holds. 

We now combine the information above from Case C and Case N. 

Note that $L(f) \ge 2$ (due to $x^2 \nmid f$ and $\deg f \ge 3$). 
If $L(f)=2$, then $f(x)=\pm x^d \pm x$ or $f(x)=\pm x^d \pm 1$, 
and so we can choose $g(x)=\pm x^d \pm 2x$ or $g(x)=\pm x^d \pm 2$ accordingly for our purpose.    
Hence, in all what  follows,  we assume that $L(f) \ge 3$. 

Let us put 
$$
m =  \lfloor d+2d\Big( \frac{\log d}{\log\log d} \Big)^{3} \log L(f) \rfloor. 
$$
Combining Case C with Case N and using \eqref{eq:CaseN}, we 
derive that for some integer 
$$
n\in \{m+1,m+2,\ldots,m+\Phi(d), m+\Phi(d)+1\},
$$
the polynomial $g(x)=x^n+f(x)$ is square-free. 
It remains to bound the degree of $g$, that is, $n$. 
Clearly, 
\begin{equation*}
\begin{split}
n & \le m+\Phi(d)+1\\
& \le d+2d\Big( \frac{\log d}{\log\log d} \Big)^{3} \log L(f) +\Phi(d)+1. 
\end{split}
\end{equation*}

Therefore, in order to get the desired upper bound \eqref{baab}, it 
suffices to establish the following inequality: 
\begin{equation}\label{njinji}
d+ \Phi(d)+1 < 0.2d\Big( \frac{\log d}{\log\log d} \Big)^{3} \log L(f). 
\end{equation}

For  $3 \le d \le 10^6$, by Lemma \ref{lem:Phi}, one has $d+\Phi(d) \le 3.5d$. 
Besides, we have the inequality $(\log d/\log\log d)^{3} > 20$ for any $d \ge 3$. 
Hence,
$$
0.2d\Big( \frac{\log d}{\log\log d} \Big)^{3} \log L(f) > 0.2d \cdot 20\log 3 > 4.3d > d+ \Phi(d)+1,
$$
which yields \eqref{njinji} for  $3 \le d \le 10^6$.

Next, for $d > 10^6$ the inequality $(\log d/\log\log d)^{3} \log L(f) > 160$ is true by noticing $L(f) \ge 3$.  
Furthermore, by Lemma \ref{lem:Phi}, we have $d+\Phi(d) \le 24d$. Combining these inequalities, we deduce that
$$
0.2d\Big( \frac{\log d}{\log\log d} \Big)^{3} \log L(f) > 32d > d+ \Phi(d)+1. 
$$
This proves \eqref{njinji} for every integer $d > 10^6$, and so completes the proof of the case when $x^2 \nmid f(x)$. 

Finally, assume that $x^2 \mid f(x)$. 
Then, the desired result follows by applying the same argument as above to the polynomial $f(x)+1$.  
\end{proof}

\section{Approaches via binary polynomials}
\label{sec:binary}

In this section, we obtain some partial results towards Conjecture \ref{conj:Turan} by transfering our problem to the setting of binary polynomials. 
This is based on the simple fact that for any integer polynomial with odd leading coefficient, if its reduction modulo 2 is square-free, then the polynomial itself is also square-free. 

Let $\F_2$ denote the binary field. 
For any polynomial $f \in \F_2[x]$, we define its \textit{length} $L_2(f)$ to be the number of its monomials. 
For a polynomial $f(x) =\sum_{i=0}^{d} a_i x^i \in \F_2[x]$, where $a_i \in \{0,1\}$, of degree $d \ge 2$, we define 
$$
f_e(x) = \sum_{i=0}^{\lfloor d/2 \rfloor} a_{2i} x^i,  \qquad f_o(x) = \sum_{i=0}^{\lfloor (d-1)/2 \rfloor} a_{2i+1} x^i. 
$$
Clearly, we have $f(x)= f_e(x)^2 + xf_o(x)^2$ and the derivative
satisfies $f^\prime(x) = f_o(x)^2$.

We first present a simple but useful lemma. 

\begin{lemma}   \label{lem:fefo}
For any polynomial $f \in \F_2[x]$ of degree at least $2$, $f$ is square-free if and only if $\gcd(f_e,f_o)=1$. 
Moreover, any multiple root of $f$ is a root of the polynomial $\gcd(f_e,f_o)$.
\end{lemma}

\begin{proof}
Note that $f$ is square-free if and only if $\gcd(f,f^\prime)=1$. 
We see that this is equivalent to $\gcd(f_e^2 + xf_o^2,f_o^2)=1$, 
that is, $\gcd(f_e^2,f_o^2)=1$.  This happens if and only if $\gcd(f_e,f_o)=1$. 
The other statement can be obtained similarly. 
\end{proof}

Based on Lemma \ref{lem:fefo}, we can use PARI/GP \cite{PARI} to test binary polynomials of low degree. 
Our calculations show the following: 

\begin{lemma}  \label{lem:36}
For each polynomial $f \in \F_2[x]$ of degree $d \le 36$ which is not square-free and
satisfies $f(0) \ne 0$, 
there exists an integer $n$ with $0 < n < d$ such that $x^n + f(x)$ is square-free.   
\end{lemma}

\begin{corollary}
For each polynomial $f \in \F_2[x]$ of degree $d \le 37$ which is not square-free and satisfies $x \mid f $ and $ x^2 \nmid f$, 
there exists an integer $n$ with $1 < n < d$ such that $x^n + f(x)$ is square-free.   
\end{corollary} 

\begin{proof}
The result follows by applying Lemma \ref{lem:36} to the polynomial $f(x)/x$. 
\end{proof}

We know from \cite[Section 2]{FM} that if $f\in \F_2[x]$ has degree $d \le 40$ and satisfies $f(0) \ne 0$, 
then there is an irreducible polynomial $g \in \F_2[x]$ with degree $d$ and $L_2(f-g) \le 3$. 
Using this, we can handle polynomials of higher degree. 

\begin{lemma}  \label{lem:81}
For any polynomial $f \in \F_2[x]$ of degree $d \le 81$ satisfying $f(0) \ne 0$,  
there exists a square-free polynomial $g \in \F_2[x]$ of degree $d$ such that $L_2(f-g) \le 3$.
\end{lemma} 

\begin{proof}
We prove the desired result case by case. 

Since $f(0) \ne 0$, we have $f_e \ne 0$ and $f_e(0) \ne 0$. 
If $f_o = 0$ (that is, $f=f_e^2$), then we choose $g(x) = f_e^2(x) + x$, 
by Lemma~\ref{lem:fefo} $g$ is square-free (because $g_o=1$), and also $L_2(f-g)=1$.  

In the sequel, assume that $f_o \ne 0$, and write $f_o = x^kf_1$ with $f_1(0) \ne 0$.  

If $\deg f_e > \deg f_1$, then by the above mentioned result, 
there is an irreducible polynomial $h \in \F_2[x]$ with degree $\deg f_e$ and $L_2(f_e-h) \le 3$, 
which also satisfies $h(0) \ne 0$. 
Since $\deg h > \deg f_1$, we have $\gcd(h,f_1)=1$, and so $\gcd(h,f_o)=1$. 
We then choose $g(x) = h^2(x) + xf_o^2(x)$. Then, by Lemma~\ref{lem:fefo}, $g$ is square-free, and also 
$L_2(f-g) = L_2(f_e-h) \le 3$. 

By symmetry, one can settle the case when  $\deg f_1 > \deg f_e$
in a similar fashion. 

Finally, we assume that $\deg f_e = \deg f_1$.  
As the above, there is an irreducible polynomial $h \in \F_2[x]$ with degree $\deg f_e$, $h(0) \ne 0$ and $L_2(f_e-h) \le 3$. 
If $f_1$ is reducible, we choose $g(x) = h^2(x) + xf_o(x)^2$. Then, by Lemma~\ref{lem:fefo}, $g$ is square-free (since $\gcd(h,f_o)=1$), 
and also $L_2(f-g) \le 3$. 
If otherwise $f_1$ is irreducible, then to complete the proof one can choose  $g= f$ or $g=f+x^2$, because 
$f_e$ and $f_e + x$ are coprime ($f_e(0) \ne 0$) and so at least one of them is coprime to $f_o$.  
\end{proof}

By adding some extra conditions, one can include more polynomials.

\begin{lemma}   \label{lem:ffefo}
For any polynomial $f \in \F_2[x]$ of degree $d \ge 9$, assume that one of the following two conditions holds: 
\begin{itemize}
\item $f_e$ is not divisible by  $x, x+1$ and $x^2+x+1$, and $f_o$ has at most $5$ distinct irreducible factors; 
\item $f_o$ is not divisible by  $x, x+1$ and $x^2+x+1$, and $f_e$ has at most $5$ distinct irreducible factors. 
\end{itemize} 
Then, there exists a square-free polynomial $g \in \F_2[x]$ of degree $d$ such that $L_2(f-g) \le 1$. 
\end{lemma}

\begin{proof} 
By symmetry, we only need to prove the case when the first condition holds. 
Since $f_e$ is not divisible by  $x, x+1$ and $x^2+x+1$, the following non-zero polynomials $$f_e(x), \> f_e(x)+1, \> f_e(x) + x, \> f_e(x) + x^2, \> f_e(x) + x^3, \> f_e(x) + x^4$$ 
are pairwise coprime. 
If $f_o=0$, then we can choose $g(x)=f_e(x)^2 + x$. By Lemma \ref{lem:fefo}, such $g$ is square-free. 

Next, assume that $f_o \ne 0$. 
Under the assumption that $f_o$ has at most $5$ distinct irreducible factors, we deduce that there is a polynomial of the form $f_e+h$, where $h=0$ or $x^j, 0 \le j \le 4$, 
such that $\gcd(f_o,f_e+h)=1$. 
Then, we can choose $$g(x) = (f_e(x)+h(x))^2 + xf_o(x)^2.$$ 
By Lemma \ref{lem:fefo}, $g$ is square-free. Moreover, it is clear that $\deg g = d$ and $L_2(f-g) \le 1$. 
This completes the proof of the lemma. 
\end{proof}

Now, we can use the above results to record some partial progress towards Conjecture \ref{conj:Turan}. 
The following theorem is a direct consequence of the above results. 

\begin{theorem}    \label{thm:binary}
We have the following:
\begin{itemize}
\item for any polynomial $f \in \Z[x]$ of degree $d \le 36$ with odd leading and constant coefficients, 
 there exists a square-free polynomial $g \in \Z[x]$ of degree $d$ 
such that $L(f-g) \le 1$; 
 
 \item for any polynomial $f \in \Z[x]$ of degree $d \le 37$ with odd leading coefficient and even constant term and 
 such that $0$ is a simple root of the reduction of $f$ modulo $2$, 
 there exists a square-free polynomial $g \in \Z[x]$ of degree $d$ satisfying $L(f-g) \le 1$; 
 
\item for any polynomial $f \in \Z[x]$ of degree $d \le 81$ with odd leading and constant coefficients, 
 there exists a square-free polynomial $g \in \Z[x]$ of degree $d$ such that $L(f-g) \le 3$; 

\item for any polynomial $f \in \Z[x]$ of degree $d \ge 9$ with odd leading coefficient and such that
the reduction of $f$ modulo $2$ satisfies one of the two conditions in Lemma \ref{lem:ffefo},
 there exists a square-free polynomial $g \in \Z[x]$ of degree $d$ satisfying $L(f-g) \le 1$. 
\end{itemize}
\end{theorem}

Note that from Theorem \ref{thm:binary} one can obtain various classes of polynomials $f \in \Z[x]$ 
such that  there exists a square-free polynomial $g \in \Z[x]$ of degree $\deg f$ satisfying $L(f-g) \le 2$.

\section{Polynomials over prime fields}
\label{sec:prime}

In this section, we consider polynomials over prime fields. 
Let $\F_p$ be the finite field with $p$ elements, where $p$ is a prime number. 
For any polynomial $f \in \F_p[x]$, define its \textit{length} $L_p(f)$ by choosing each of its coefficients 
in the interval $(-p/2,p/2]$ and then summing their absolute values (in $\Z$). 
We want to show that there is a positive proportion of polynomials in $\F_p[x]$ whose distance to square-free polynomials is at least 2. 
We remark that the distance to irreducible polynomials over prime fields has been considered in  \cite[Theorem 2]{FM} and \cite[Section 6]{filas}.  

Let $N_p(d)$ be the number of polynomials $f$ in $\F_p[x]$ of degree $d$ such that $L_p(f-g) \ge 2$ for any square-free polynomial 
$g \in \F_p[x]$. 

\begin{theorem}
We have the following: 
\begin{itemize}
\item for any $d \ge 8$, we have $N_2(d) \ge 2^{d-8}$;

\item for any $d \ge 14$, we have $N_3(d) \ge 2 \cdot 3^{d-14}$;

\item for any prime number $p \ge 5$, we have $N_p(15) \ge (p-2)p^5$, 
and for any integer $d \ge 16$, $N_p(d) \ge (p-1)p^{d-10}$.
\end{itemize}
\end{theorem}

\begin{proof}
We first handle the case $p=2$. 
Consider the polynomials $f \in \F_2[x]$ of the form 
\begin{equation*}   
f(x) = x^2(x+1)^2(x^2+x+1)^2u(x) + x^6+x^5+x^4+x^3+x^2. 
\end{equation*} 
Then, $x^2 \mid f(x)$, $(x+1)^2 \mid f(x)+1$ and $(x^2+x+1)^2 \mid f(x)+x$. 
So, for any square-free polynomial $g \in \F_2[x]$, we have $L_2(f-g) \ge 2$. 
If $d \ge 8$, we choose $u$ to be any polynomial in $\F_2[x]$ of degree $d-8$, 
and as a result, we obtain  the desired result, since there are $2^{d-8}$ possibilities to choose such $u \in \F_2[x]$. 

Next, let us consider the case when $p=3$. 
Let $f \in \F_3[x]$ be of the form $x^2h(x)$ with $h$ non-zero and $\deg f=d \ge 14$. 
Assume that $(x+1)^2$ divides $f(x)+1$, $(x-1)^2$ divides $f(x)-1$, 
$(x^2+x-1)^2$ divides $f(x)+x$, and $(x^2-x-1)^2$ divides $f(x)-x$. 
Using computations with PARI/GP \cite{PARI}, we obtain 
\begin{equation*}   
\begin{split}
f(x) = & x^2(x+1)^2(x-1)^2(x^2+x-1)^2(x^2-x-1)^2u(x)+ \\
 & x^{13} - x^{12} - x^9 +x^8+x^6 + x^5 - x^2
\end{split}
\end{equation*} 
for some polynomial $u \in \F_3[x]$ of degree $d-14$. 
There are $2 \cdot 3^{d-14}$ of such polynomials $u$, which implies 
the desired result. 

Finally, let us consider the case $p \ge 5$. 
We first choose $f(x)$ as in \eqref{eq:example}. 
We have known that $1/2, -1/2, 1/6$ and $-1/6$ are multiple roots of $f(x)-x, f(x)+1, f(x)+x$ and $f(x)-1$, respectively.  
Since $p \ge 5$, the reductions of $\pm 1/2$ and $\pm 1/6$ modulo $p$ are pairwise distinct. 
Then, viewing $f$ as a polynomial in $\F_p[x]$, we consider the polynomials: 
$$
f(x) + x^2 (2x+1)^2 (2x-1)^2 (6x+1)^2 (6x-1)^2 u(x) \in \F_p[x], 
$$
where $u \in \F_p[x]$ is of degree $d-10$. 
Then, for $d \ge 16$, the desired result follows by noticing $\deg f = 15$ and taking any polynomial $u \in \F_p[x]$ of degree $d-10$, 
since  there are  $(p-1)p^{d-10}$ of such polynomials $u$.  
When $d=15$, to ensure the considered polynomials are of degree 15, we only have $N_p(15) \ge (p-2)p^5$.  
\end{proof}

In conclusion, as an analogue of Conjecture \ref{conj:Turan},  we pose the following question. 

\begin{question}   \label{conj:Turan-p}
Does for any prime number $p$ and any polynomial $f \in \F_p[x]$, 
there exist a square-free  polynomial $g \in \F_p[x]$ of degree at most $\deg f$ satisfying $L_p(f-g) \le 2$?  
\end{question}

In Lemma \ref{lem:36} we actually give a positive answer to Question~\ref{conj:Turan-p} for polynomials in $\F_2[x]$ of degree at most $36$. Indeed, for
$f \in \F_2[x]$ with $f(0)=0$ we can replace $f(x)$ by $f(x)+1$ and then select
$g(x)=f(x)+1$ if $f(x)+1$ is square-free or, if it is not square-free,
take $g(x)=x^n+f(x)+1$
with some integer $n \in \{1,\dots,d-1\}$ for which $x^n+f(x)+1$ is square-free. 

In addition, we remark that, by a recent result of Oppenheim and Shusterman \cite[Theorem 1.2]{OS}, 
for any polynomial $f \in \F_p[x]$ of degree $d \ge 2$, there exists a square-free polynomial $g$ of degree $d$ 
such that $L_p(f-g) \leq 2(d-1)$.

\section*{Acknowledgements}

The authors are grateful to the referee who carefully read the paper and gave several valuable comments. 
The research of A.~D. was funded by a grant (No. S-MIP-17-66/LSS-110000-1274) 
from the Research Council of Lithuania.  
The research of M.~S. was supported by the Macquarie University Research Fellowship.

\end{document}